\documentclass[12pt]{amsart}
\usepackage{epsfig,color}

\theoremstyle{definition}

\newtheorem{theorem}{Theorem}[section]
\newtheorem{definition}[theorem]{Definition}
\newtheorem{conjecture}[theorem]{Conjecture}
\newtheorem{proposition}[theorem]{Proposition}
\newtheorem{lemma}[theorem]{Lemma}
\newtheorem{remark}[theorem]{Remark}
\newtheorem{corollary}[theorem]{Corollary}
\newtheorem{example}[theorem]{Example}
\newtheorem*{acknowledgements}{Acknowledgements}
\numberwithin{equation}{section}

\newcommand{\abs}[1]{\lvert#1\rvert}

\begin{document}

\title{Gap and rigidity theorems of  $\lambda$-hypersurfaces}

\author{Qiang Guang}

\address{MIT, Department of Mathematics,
77 Massachusetts Avenue, Cambridge, MA 02139-4307}

\curraddr{}
\email{qguang@math.mit.edu}   
\thanks{}

%\subjclass[2000]{53C44}

%\keywords{$\lambda$-hypersurfaces, $\lambda$-curves, the second fundamental form, Bernstein problem}

\begin{abstract}
We study $\lambda$-hypersurfaces that are critical points of a Gaussian weighted area functional $\int_{\Sigma} e^{-\frac{|x|^2}{4}}dA$ for compact variations that preserve weighted volume. First, we prove various gap and rigidity theorems  for complete $\lambda$-hypersurfaces in terms of the norm of the second fundamental form $|A|$. Second, we show that in one dimension, the only smooth complete and embedded $\lambda$-hypersurfaces in $\mathbb{R}^2$ with $\lambda\geq 0$ are lines and round circles. Moreover, we establish a Bernstein type theorem for $\lambda$-hypersurfaces which states that smooth $\lambda$-hypersurfaces that are entire graphs with polynomial volume growth are hyperplanes. All the results can be viewed as generalizations of results for self-shrinkers.
\end{abstract}

\maketitle

\section{Introduction}
We follow the notation of \cite{CW1} and call a hypersurface ${\Sigma}^{n} \subset \mathbb{R}^{n+1}$ a $\lambda$-hypersurface if it satisfies
\begin{equation}
H-\frac{\langle x,\mathbf{n}\rangle}{2}=\lambda,
\end{equation}
where $\lambda$ is any constant, $H$ is the mean curvature, $\mathbf{n}$ is the outward pointing unit normal and $x$ is the position vector.

$\lambda$-hypersurfaces were first studied by McGonagle and Ross in \cite{MR1}, where they investigate the following isoperimetric type problem in a Gaussian weighted Euclidean space:

Let $\mu(\Sigma)$ be the weighted area functional defined by $\mu(\Sigma)=\int_{\Sigma} e^{-\frac{|x|^2}{4}}dA$ for any hypersurface $\Sigma^n \subset\mathbb{R}^{n+1}$.  Consider the variational problem of minimizing $\mu(\Sigma)$ among all $\Sigma$ enclosing a fixed Gaussian weighted volume. Note that the variational problem is not to consider $\Sigma$ enclosing a specific fixed weighted volume, but to consider variations that preserve the weighted volume. 

It turns out that critical points of this variational problem are $\lambda$-hypersurfaces and the only smooth  stable ones are hyperplanes; see \cite{MR1}.

In \cite{CW1}, Cheng and Wei introduced the notation of $\lambda$-hypersurfaces by studying the weighted volume-preserving mean curvature flow. They proved that $\lambda$-hypersurfaces are critical points of the weighted area functional for the weighted volume-preserving variations. Moreover, they defined a $F$-functional of $\lambda$-hypersurfaces and studied $F$-stability, which extended a result of Colding-Minicozzi \cite{CM1}.

\begin{example}
We give three examples of $\lambda$-hypersurfaces in $\mathbb{R}^3$.
\begin{itemize}
\item[(1)] The sphere $\mathbb{S}^2(r)$ with radius $r=\sqrt{\lambda^2+4}-\lambda$.
\item[(2)] The cylinder $\mathbb{S}^1(r) \times \mathbb{R}$, where $\mathbb{S}^1(r)$ has radius $\sqrt{\lambda^2+2}-\lambda$.
\item[(3)] The hyperplane in $\mathbb{R}^3$.
\end{itemize}
\end{example}

Note that when $\lambda=0$, $\lambda$-hypersurfaces are just self-shrinkers and they can be viewed as a generalization of self-shrinkers in some sense. 

It is well-known that self-shrinkers play a key role in the study of mean curvature flow (``MCF"), since they describe the singularity models of the MCF. In one dimension, smooth complete embedded self-shrinking curves are totally understood and they are just lines and round circles by the work of Abresch and Langer \cite{AbL}. In higher dimensions, self-shrinkers are more complicated and there are only few examples; see \cite{A}, \cite{KM}, \cite{M1} and \cite{N1}. There are some classification and rigidity results of self-shrinkers under certain assumptions. Ecker and Huisken \cite{EH1} proved that if a self-shrinker is an entire graph with polynomial volume growth, then it is a hyperplane. Later, Wang \cite{WL} removed the condition of polynomial volume growth.  In {\cite{CM1}}, Colding and Minicozzi proved that the only smooth complete embedded self-shrinkers with polynomial volume growth and $H\geq 0$ in $\mathbb{R}^{n+1}$ are generalized cylinders $\mathbb{S}^k \times \mathbb{R}^{n-k}$.

\vspace{3mm}
In this paper, we study $\lambda$-hypersurfaces from three aspects: gap and rigidity results, one-dimensional case and entire graphic case. 

The first main result is the following gap theorem for $\lambda$-hypersurfaces in terms of the norm of the second fundamental form $|A|$.

\begin{theorem}\label{1.1}
If $\Sigma^n \subset \mathbb{R}^{n+1}$ is a smooth complete embedded $\lambda$-hypersurface satisfying $H-\frac{\langle x,\mathbf{n}\rangle}{2}=\lambda$ with polynomial volume growth, which satisfies
\begin{equation}\label{1.2}
|A| \leq \frac{\sqrt{\lambda^2+2}-\abs{\lambda}}{2},
\end{equation}
then $\Sigma$ is one of the following:
\begin{itemize}
\item[(1)] a round sphere $\mathbb{S}^n$,
\item[(2)] a cylinder $\mathbb{S}^k \times \mathbb{R}^{n-k}$ for $1\leq k \leq n-1$,
\item[(3)] a hyperplane in $\mathbb{R}^{n+1}$.
\end{itemize}

\end{theorem}

\begin{remark}
Note that when $\lambda=0$, then $\Sigma$ is a self-shrinker satisfying $|A|^2\leq 1/2$. So this theorem implies the gap theorem of Cao and Li \cite{CL1} in codimension one case. Cheng, Ogata and Wei \cite{COW1} obtained a gap theorem for $\lambda$-hypersurfaces in terms of $|A|$ and $H$, which also generalized Cao and Li's result.
\end{remark}

Partially motivated by the work of Chern, do Carmo and Kobayashi \cite{CDK} on minimal submanifolds of a sphere with the second fundamental form of constant length, we consider smooth closed embedded $\lambda$-hypersurfaces $\Sigma^2\subset \mathbb{R}^3$ with $\abs{A}=constant$ and $\lambda \geq 0$. We prove that they are just round spheres. It can be thought of as a generalization of the result that any smooth self-shrinker in $\mathbb{R}^3$ with $|A|=constant$ is a generalized cylinder; see \cite{DX1} and \cite{G1}. 

\begin{theorem}\label{1.4}
Let $\Sigma^2 \subset \mathbb{R}^3$ be a smooth closed and embedded $\lambda$-hypersurface with $\lambda \geq 0$. If the second fundamental form of $\Sigma^2$ is of constant length, i.e., $|A|=constant$, then $\Sigma^2$ is a round sphere.
\end{theorem}
The proof of this theorem has two ingredients. The first ingredient is to consider the point where the norm of the position vector $|x|$ achieves its minimum.  This will give that the genus is $0$. The second ingredient is an interesting result from \cite{HW1} that any smooth closed special $W$-surface of genus $0$ is a round sphere. 

\vspace{3mm}
Next, we turn to the one-dimensional case. Following an argument in \cite{MC1}, we show that just as self-shrinkers in $\mathbb{R}^2$, the only smooth complete and embedded $\lambda$-hypersurfaces ($\lambda$-curves) in $\mathbb{R}^2$  with $\lambda \geq 0$ are lines and round circles.

\begin{theorem}\label{1.7}
Any smooth complete embedded $\lambda$-hypersurface  ($\lambda$-curve) $\gamma$  in $\mathbb{R}^2$  satisfying $H-\frac{\langle x,\mathbf{n}\rangle}{2}=\lambda$ with $\lambda \geq0$  must either be a line or a round circle.
\end{theorem}

In contrast to embedded self-shrinking curves, the dynamical pictures suggest that there exist some embedded $\lambda$-curves with $\lambda<0$ which are not round circles. There also exist Abresch-Langer type curves for immersed $\lambda$-curves; see \cite{CH1} for more details.

\vspace{3mm}
In the last part, we give a Bernstein type theorem for $\lambda$-hypersurfaces, which generalizes Ecker and Huisken's result \cite{EH1}.

\begin{theorem}\label{1.5}
If a $\lambda$-hypersurface $\Sigma^n \subset \mathbb{R}^{n+1}$ is an entire graph with polynomial volume growth satisfying $H-\frac{\langle x,\mathbf{n}\rangle}{2}=\lambda$, then $\Sigma$ is a hyperplane.

\end{theorem}

\begin{remark}
A similar result is also obtained later by Cheng and Wei \cite{CW2} under the assumption of properness instead of polynomial volume growth. Note that they proved properness of $\lambda$-hypersurfaces  implies polynomial volume growth; see Theorem 9.1 in \cite{CW1}.
\end{remark}

Part of the reason we are interested in $\lambda$-hypersurfaces is that closed $\lambda$-hypersurfaces ($\lambda\geq 0$) behave nicely under the rescaled mean curvature flow; see the discussion in the Appendix.

\begin{acknowledgements}
The author would like to thank Professor William Minicozzi for his valuable and constant support.
\end{acknowledgements}

\section{Background and Preliminaries}
In this section, we recall some background and collect several useful formulas for $\lambda$-hypersurfaces. Throughout this paper, we always assume hypersurfaces to be smooth complete embedded, without boundary and with polynomial volume growth.

\subsection{Notion and conventions} 
Let $\Sigma \subset \mathbb{R}^{n+1}$ be a hypersurface. Then $\nabla_{\Sigma}$, div, and $\Delta$ are the gradient, divergence, and Laplacian, respectively, on $\Sigma$. $\mathbf{n}$ is the outward unit normal, $H$=div$_{\Sigma}\mathbf{n}$ is the mean curvature, $A$ is the second fundamental form, and $x$ is the position vector. With this convection, the mean curvature $H$ is $n/r$ on the sphere $\mathbb{S}^n\subset \mathbb{R}^{n+1}$ of radius $r$. If $e_i$ is an orthonormal frame for $\Sigma$, then the coefficients of the second fundamental form are defined to be $a_{ij}=\langle \nabla_{e_i}e_j,\mathbf{n}\rangle$.

\subsection{Simons type identity}
Now we will derive a Simons type identity for $\lambda$-hypersurface $\Sigma$ which plays a key role in our proof of Theorem \ref{1.1}.
First, recall the operators $\mathcal{L}$ and $L$ from \cite{CM1} defined by 
\begin{equation}
\mathcal{L}=\Delta-\frac{1}{2}\langle x, \nabla \cdot \rangle,
\end{equation}
\begin{equation}
L=\Delta-\frac{1}{2}\langle x, \nabla \cdot \rangle+\abs{A}^2+\frac{1}{2}.
\end{equation}

\begin{lemma}\label{2.10}
If $\Sigma^n\subset \mathbb{R}^{n+1}$ is a $\lambda$-hypersurface satisfying $H-\frac{\langle x,\mathbf{n}\rangle}{2}=\lambda$, then
\begin{equation}\label{2.11}
LA=A-\lambda A^2,
\end{equation}

\begin{equation}\label{2.12}
LH=H+\lambda |A|^2,
\end{equation}

\begin{equation}\label{2.13}
\mathcal{L}|A|^2=2\Big(\frac{1}{2}-|A|^2\Big)|A|^2-2\lambda\langle A^2,A\rangle+2|\nabla A|^2.
\end{equation}
\end{lemma}

\begin{remark}
More general results of above formulas were already obtained by Colding and Minicozzi; see Proposition 1.2 in \cite{CM16}. For completeness we also include a proof here. Note that when $\lambda=0$, these formulas are just Simons' equations for self-shrinkers in \cite{CM1}. 
\end{remark}

\begin{proof}[Proof of Lemma \ref{2.10}]
Recall that for a general hypersurface, the second fundamental form $A$ satisfies 
\begin{equation}\label{2.15}
\Delta A=-|A|^2 A-HA^2-Hess_{H}.
\end{equation}
Now we fix a point $p\in \Sigma$, and choose a local orthonormal frame $e_i$ for $\Sigma$ such that its tangential covariant derivatives vanish. So at this point, we have $\nabla_{e_i}e_j=a_{ij}\mathbf{n}$. Thus,
\begin{equation}\label{2.16}
\begin{split}
2Hess_H(e_i,e_j)&=\nabla_{e_j}\nabla_{e_i}\langle x,\mathbf{n}\rangle =\langle x, -a_{ik}e_k \rangle_j \\&=-a_{ikj}\langle x, e_k \rangle-a_{ij}-a_{ik}a_{jk}\langle x,\mathbf{n}\rangle \\ &=-(\nabla_{x^{T}} A)(e_i,e_j)-A(e_i,e_j)-\langle x,\mathbf{n}\rangle A^2(e_i,e_j).
\end{split}
\end{equation}
Combining (\ref{2.15}) with (\ref{2.16}) gives
\begin{equation}
LA=\Delta A-\frac{1}{2}\nabla_{x^{T}}A+\Big(\frac{1}{2}+|A|^2\Big)A=A-\Big(H-\frac{\langle x,\mathbf{n}\rangle}{2}\Big)A^2=A-\lambda A^2.
\end{equation}
This gives (\ref{2.11}) and taking the trace gives (\ref{2.12}). 
For (\ref{2.13}), we have that
\begin{equation}\label{2.17}
\begin{split}
\mathcal{L}|A|^2 & = 2\langle \mathcal{L}A,A \rangle + 2 |\nabla A|^2 \\& = 2 |A|^2 - 2\lambda\langle A^2,A\rangle-2\Big(\frac{1}{2}+ |A|^2 \Big)|A|^2 +2 |\nabla A|^2 \\& = 2\Big(\frac{1}{2}-\abs{A}^2\Big)\abs{A}^2-2\lambda\langle A^2,A \rangle + 2 |\nabla A|^2.
\end{split}
\end{equation}
This completes the proof.
\end{proof}

\subsection{Weighted integral estimates for $\abs{A}$}
In this subsection, we prove a result which will justify our integration when hypersurfaces are non-compact and with bounded $|A|$.

\begin{proposition}\label{2.20}
If $\Sigma^n\subset \mathbb{R}^{n+1}$ is a complete $\lambda$-hypersurface with polynomial volume growth satisfying $|A|\leq C_0$, then 
\begin{equation}\label{2.21}
\int_{\Sigma} |\nabla A|^2 e^{-\frac{|x|^2}{4}}<\infty.
\end{equation}
\end{proposition}

The proof of Proposition \ref{2.20} relies on the following two lemmas from \cite{CM1} which show that the linear operator $\mathcal{L}$ is self-adjoint in a weighted $L^2$ space.

\begin{lemma}[\cite{CM1}]\label{2.23}
If $\Sigma\subset\mathbb{R}^{n+1}$ is a hypersurface, $u$ is a $C^1$ function with compact support, and $v$ is a $C^2$ function, then
\begin{equation}\label{2.24}
\int_{\Sigma} u(\mathcal{L}v)e^{-\frac{|x|^2}{4}}=-\int_\Sigma \langle \nabla u, \nabla v\rangle e^{-\frac{|x|^2}{4}}.
\end{equation}
\end{lemma}

\begin{lemma}[\cite{CM1}]\label{2.25}
Suppose that $\Sigma\subset\mathbb{R}^{n+1}$ is a complete hypersurface without boundary. If $u,v$ are $C^2$ functions with
\begin{equation}\label{2.26}
\int_\Sigma \big( |u\nabla v| + |\nabla u||\nabla v|+|u\mathcal{L}v| \big) e^{-\frac{|x|^2}{4}}<\infty,
\end{equation}
then we get
\begin{equation}\label{2.27}
\int_{\Sigma} u(\mathcal{L}v)e^{-\frac{|x|^2}{4}}=-\int_\Sigma \langle \nabla u, \nabla v\rangle e^{-\frac{\abs{x}^2}{4}}.
\end{equation}
\end{lemma}

\begin{proof}[Proof of Proposition \ref{2.20}]
By Lemma \ref{2.10} and $|A|\leq C_0$, we have
\begin{equation}\label{2.30}
\begin{split}
\mathcal{L}|A|^2 & = 2\Big(\frac{1}{2}-|A|^2\Big)|A|^2 - 2\lambda\langle A^2,A \rangle + 2|\nabla A|^2 \\& \geq 2\Big(\frac{1}{2}-|A|^2 \Big)|A|^2-2\abs{\lambda}|A|^3 + 2|\nabla A|^2  \geq 2|\nabla A|^2 - C,
\end{split}
\end{equation}
where $C$ is a positive constant depending only on $\lambda$ and $C_0$. We  allow $C$ to change from line to line. 
For any smooth function $\phi$ with compact support, we integrate (\ref{2.30}) against $\frac{1}{2}\phi^2$. 

By Lemma \ref{2.23}, we obtain that
\begin{equation}
-2\int_\Sigma \phi|A|\langle \nabla \phi, \nabla |A|\rangle e^{-\frac{|x|^2}{4}} \geq \int_\Sigma \phi^2( |\nabla A|^2 - C )e^{-\frac{|x|^2}{4}}.
\end{equation}  
Using the absorbing inequality $\epsilon a^2+\frac{b^2}{\epsilon}\geq 2ab$ gives
\begin{equation}\label{2.34}
\int_\Sigma (\epsilon \phi^2 |\nabla |A||^2 +\frac{1}{\epsilon}|A|^2 |\nabla \phi|^2) e^{-\frac{|x|^2}{4}}\geq \int_\Sigma \phi^2(|\nabla A|^2 - C ) e^{-\frac{|x|^2}{4}}.
\end{equation}
Now we choose $\abs{\phi}\leq 1$, $\abs{\nabla \phi} \leq 1$ and $\epsilon=1/2$. Combining this with $\abs{\nabla A} \geq |\nabla |A||$, we see that (\ref{2.34}) gives
\begin{equation}
\int_\Sigma (4|A|^2+C)e^{-\frac{|x|^2}{4}}\geq \int_\Sigma \phi^2 |\nabla A|^2 e^{-\frac{|x|^2}{4}}.
\end{equation}
The conclusion follows from the monotone convergence theorem and the fact that $\Sigma$ has polynomial volume growth.

\end{proof}

A direct consequence of Proposition \ref{2.20} and Lemma \ref{2.25} is the following corollary.

\begin{corollary}\label{2.35}
If $\Sigma^n\subset \mathbb{R}^{n+1}$ is a complete $\lambda$-hypersurface with polynomial volume growth satisfying $|A| \leq C_0$, then 
\begin{equation}\label{2.36}
\int_\Sigma \mathcal{L}|A|^2 e^{-\frac{|x|^2}{4}}=0.
\end{equation}
\end{corollary}

\section{Gap Theorems for $\lambda$-hypersurfaces}

\subsection{Proof of Theorem \ref{1.1}}

Now we are ready to prove Theorem \ref{1.1}.
\begin{proof}[Proof of Theorem \ref{1.1}]
By Lemma \ref{2.10}, we have
\begin{equation}\label{3.40}
\frac{1}{2}\mathcal{L}|A|^2 = \Big(\frac{1}{2} - |A|^2 \Big) |A|^2 - \lambda\langle A^2, A \rangle + |\nabla A|^2  \geq \Big(\frac{1}{2}-|A|^2 \Big) |A|^2 - |\lambda| |A|^3 + |\nabla A|^2.
\end{equation}
Proposition \ref{2.20} and Corollary \ref{2.35}  give 
\begin{equation}\label{3.41}
0 = \int_\Sigma \mathcal{L}|A|^2 e^{-\frac{|x|^2}{4}}\geq \int_\Sigma \Big(\frac{1}{2}-|A|^2 - |\lambda| |A|\Big) |A|^2 e^{-\frac{|x|^2}{4}}+\int_\Sigma |\nabla A|^2 e^{-\frac{|x|^2}{4}}.
\end{equation}
Note that when $|A| \leq \frac{\sqrt{\lambda^2+2}-\abs{\lambda}}{2}$, we have 
\begin{equation*}
\frac{1}{2}-|A|^2-|\lambda| |A| \geq 0.
\end{equation*}
This implies the first term of (\ref{3.41}) on the right hand side is nonnegative. 
Therefore, (\ref{3.41}) implies that all inequalities are equalities. Moreover, we have
\begin{equation*}
|\nabla A|=\Big(\frac{1}{2}-|A|^2 - |\lambda||A| \Big)|A|^2=0.
\end{equation*} 
By Theorem 4 of Laswon \cite{LS69} that every smooth hypersurface with $\nabla A=0$ splits isometrically as a product of a sphere and a linear space, we finish the proof.

\end{proof}

By the proof of Theorem \ref{1.1}, we have the following gap result.

\begin{corollary}
If $\Sigma^n \subset \mathbb{R}^{n+1}$ is a smooth complete embedded $\lambda$-hypersurface satisfying $H-\frac{\langle x,\mathbf{n}\rangle}{2}=\lambda$ with polynomial volume growth, which satisfies
\begin{equation}
|A| < \frac{\sqrt{\lambda^2+2}-|\lambda|}{2},
\end{equation}
then $\Sigma$ is a hyperplane in $\mathbb{R}^{n+1}$.
\end{corollary}

\begin{remark}\label{3.42}
Note that in Theorem \ref{1.1}, when $\Sigma^n$ is a round sphere, this forces $\lambda=0$. we will address this issue in the next subsection and prove a gap theorem for closed $\lambda$-hypersurfaces with arbitrary $\lambda\geq0$.
\end{remark}

\subsection{Gap Theorems for closed $\lambda$-hypersurfaces}
We consider closed $\lambda$-hypersurfaces with $\lambda\geq 0$ in this subsection. 

\begin{lemma}\label{3.43}
If $\Sigma^n\subset \mathbb{R}^{n+1}$ is a smooth $\lambda$-hypersurface, then 
\begin{equation}
\mathcal{L}|x|^2 = 2n - |x|^2 - 2\lambda\langle x, \mathbf{n}\rangle.
\end{equation}
\end{lemma}

\begin{proof}
Recall that for any hypersurface, we have $\Delta x=-H\mathbf{n}$. Therefore,
\begin{equation}
\begin{split}
\mathcal{L}|x|^2 =\Delta |x|^2-\frac{1}{2}\langle x, \nabla |x|^2 \rangle & =2\langle \Delta x,x \rangle + 2|\nabla x|^2 - |x^{T}|^2 \\& =-2H\langle x,\mathbf{n}\rangle+2n-\abs{x^{T}}^2 \\&= 2n - |x|^2 - 2\lambda\langle x,\mathbf{n}\rangle.
\end{split}
\end{equation}
\end{proof}
A simple application of Lemma \ref{3.43} and the maximum principle give the following corollary.

\begin{corollary}\label{3.44}
Let $\Sigma^n\subset \mathbb{R}^{n+1}$ be a smooth closed $\lambda$-hypersurface with $\lambda\geq 0$. If $|x| \leq \sqrt{\lambda^2+2n}-\lambda$, then $\Sigma$ is a round sphere.
\end{corollary}

Now we are ready to address the issue mentioned in Remark \ref{3.42} and prove the gap theorem for closed $\lambda$-hypersurfaces in terms of $|A|$.

\begin{theorem}\label{3.45}
Let $\Sigma^n\subset \mathbb{R}^{n+1}$ be a smooth closed $\lambda$-hypersurface with $\lambda\geq 0$. If $\Sigma$ satisfies
\begin{equation}\label{3.46}
|A|^2\leq \frac{1}{2}+\frac{\lambda(\lambda+\sqrt{\lambda^2+2n})}{2n},
\end{equation}
then $\Sigma$ is a round sphere with radius $\sqrt{\lambda^2+2n}-\lambda$.
\begin{proof}
Since $\Sigma$ is closed, we consider the point $p$ where $|x|$ achieves its maximum. At point $p$, $x$ and $\mathbf{n}$ are in the same direction. This implies  $2H(p)=2\lambda+|x|(p)$. 

By (\ref{3.46}), we have 
\begin{equation}
\Big(\lambda+\frac{\abs{x}(p)}{2}\Big)^2=H^2(p)\leq n|A|^2 \leq n\Big(\frac{1}{2}+\frac{\lambda(\lambda+\sqrt{\lambda^2+2n})}{2n}\Big).
\end{equation}
This gives 
\begin{equation}
\max_\Sigma |x|\leq |x|(p)\leq \sqrt{\lambda^2+2n}-\lambda.
\end{equation}
By Corollary \ref{3.44}, we conclude that $\Sigma$ is a round sphere.
\end{proof}

\end{theorem}

We also include another weak gap theorem for closed $\lambda$-hypersurfaces in $\mathbb{R}^3$ in that the proof is interesting and using Gauss-Bonnet Formula, Minkowski Integral Formulas and the Willmore's inequality.
\begin{theorem}\label{3.72}
Let $\Sigma^2\subset \mathbb{R}^3$ be a closed  $\lambda$-hypersurface satisfying $H-\frac{\langle x, \mathbf{n}\rangle}{2}=\lambda$ with $\lambda\geq 0$. If $\Sigma$ satisfies $|A|^2 \leq \frac{1+\lambda^2}{2}$, then $\Sigma$ is a round sphere.

\end{theorem}

\begin{proof}
First, by the Gauss-Bonnet Theorem, we have the following identity:
\begin{equation}\label{3.73}
\int_\Sigma H^2=\int_\Sigma |A|^2 + 8 \pi(1-g),
\end{equation}
where $g$ is the genus of $\Sigma$.

Next, by Minkowski Integral Formulas and Stokes' theorem, we have
\begin{equation}\label{3.74}
\int_\Sigma H \langle x,\mathbf{n}\rangle=2\text{Area}(\Sigma),
\end{equation}
\begin{equation}\label{3.75}
\int_\Sigma \langle x,\mathbf{n}\rangle=3\text{Volume}(\Omega).
\end{equation}
Here $\Omega$ is the region enclosed by $\Sigma$.

By (\ref{3.74}), (\ref{3.75}) and the $\lambda$-hypersurface equation, we get that
\begin{equation}
\int_\Sigma H \geq \lambda \text{Area}(\Sigma)=\lambda \Big(\int_\Sigma H^2 -\lambda \int_\Sigma H\Big).
\end{equation}
This implies
\begin{equation}\label{3.76}
\int_\Sigma H\geq \frac{\lambda}{1+\lambda^2}\int_\Sigma H^2.
\end{equation}
Using (\ref{3.73}), (\ref{3.74}) and the $|A|$ bound, we see that
\begin{equation}
\int_\Sigma H^2 \leq \Big(\frac{1+\lambda^2}{2}\Big)\text{Area}(\Sigma)+8\pi (1-g) \leq \Big(\frac{1+\lambda^2}{2}\Big)\Big(\int_\Sigma H^2 -\lambda \int_\Sigma H\Big)+8\pi (1-g).
\end{equation}
Combining this with (\ref{3.76}) gives 
\begin{equation}\label{3.77}
\int_\Sigma H^2 \leq 16\pi (1-g).
\end{equation}
Recall that for any smooth closed surface $M$ in $\mathbb{R}^3$, the Willmore energy $\int_M H^2$ is greater than or equal to $16\pi$, with the equality holds if and only $M$ is a round sphere.
Therefore, by (\ref{3.77}) we conclude that $\Sigma$ is a round sphere. Actually this case implies that $\lambda=0$.

\end{proof}

\subsection{Closed $\lambda$-hypersurfaces in $\mathbb{R}^3$ with the second fundamental form of constant length}
Let $\Sigma^2\subset \mathbb{R}^3$ be a smooth complete embedded self-shrinker. If $\abs{A}$ is a constant, then one can show that $\Sigma$ is a generalized cylinder $\mathbb{S}^k\times \mathbb{R}^{2-k}$ for some $k \leq 2$; see \cite{DX1} and \cite{G1}. One way to prove this is to consider the point where the norm of position vector $|x|$ achieves its minimum.

For $\lambda$-hypersurfaces, we will use a similar idea and an important result from \cite{HW1} to show that any smooth closed $\lambda$-hypersurface in $\mathbb{R}^3$ with $\lambda\geq 0$ and $|A|=constant$ is a round sphere, i.e., Theorem \ref{1.4}.

First of all, we recall the following ingredients from \cite{HW1}.
\begin{definition}
A hypersurface in $\mathbb{R}^3$ is called a special Weingarten surface (special $W$-surface) if its Gauss curvature and mean curvature\footnote{In \cite{HW1}, they use the average rather than the sum of the principal curvatures.}, $K$ and $H$, are connected by an identity 
\begin{equation}
F(K,H)=0
\end{equation}
in which $F$ satisfies the following condition:
\begin{itemize}
\item The function $F(K,H)$ is defined and of class $C^2$ on the portion $4K\leq H^2$ of the $(K,H)-plane$ and satisfies
\begin{equation}
F_H+HF_K \neq 0 \,\,\, \,\,\,\,  {\text{when}} \,\,\, \,\, 4K=H^2.
\end{equation}
\end{itemize}
\end{definition} 

In \cite{HW1}, Hartman and Wintner proved the following theorem for special $W$-surfaces.

\begin{theorem}[\cite{HW1}]\label{3.80}
If a closed orientable surface $S$ of genus 0 is a special $W$-surface of class $C^2$, then $S$ is a round sphere.

\end{theorem}

One may easily check that a closed surface with $\abs{A}=constant$ is a special $W$-surface, so by Theorem \ref{3.80}, we have the following corollary.

\begin{corollary}\label{3.81}
Let $\Sigma^2\subset \mathbb{R}^3$ be a smooth closed embedded surface of genus 0. If $|A|=constant$, then $\Sigma$ is a round sphere.
\end{corollary}

By Corollary \ref{3.81}, in order to prove Theorem \ref{1.4}, all we need to show is that any closed $\lambda$-hypersurface with constant $|A|$ has genus 0.

\begin{proof}[Proof of Theorem \ref{1.4}]
First, by Gauss-Bonnet Formula, Minkowski Integral Formulas and Stokes' theorem, we have
\begin{equation}\label{3.82}
\int_\Sigma H^2=\int_\Sigma |A|^2+8 \pi(1-g),
\end{equation}
\begin{equation}\label{3.83}
\int_\Sigma H \langle x,\mathbf{n}\rangle=2\text{Area}(\Sigma),
\end{equation}

\begin{equation}\label{3.84}
\int_\Sigma \langle x,\mathbf{n}\rangle=3\text{Volume}(\Omega),
\end{equation}
where $g$ is the genus of $\Sigma$ and $\Omega$ is the region enclosed by $\Sigma$.

Combining above identities, we deduce that
\begin{equation}\label{3.85}
\int_\Sigma H^2 \geq (\lambda^2+1)\int_\Sigma=(\lambda^2+1)\text{Area}(\Sigma).
\end{equation}
Next, we consider the point $p\in\Sigma$ where $|x|$ achieves its minimum.
By Lemma \ref{3.43}, at point $p$, we have
\begin{equation}\label{3.86}
H^2(p)\leq \frac{2+\lambda^2+\lambda \sqrt{\lambda^2+4}}{2}.
\end{equation}
At point $p$, we can choose a local orthonormal frame $\{e_1,e_2\}$ such that the second fundamental form $a_{ij}=\lambda_i \delta_{ij}$ for $i,j=1,2$. Thus, we have 
\begin{equation}\label{3.87}
|\nabla H|^2=(a_{111}+a_{221})^2+(a_{112}+a_{222})^2.
\end{equation}
Since  $|A|^2=constant$, we see that
\begin{equation}\label{3.88}
a_{11}a_{111}+a_{22}a_{221}=a_{11}a_{112}+a_{22}a_{222}=0.
\end{equation}
Note that at point $p$, $\abs{\nabla H}=0$. This implies 
\begin{equation*}
a_{111}+a_{221}=a_{112}+a_{222}=0.
\end{equation*}
Combining this with (\ref{3.87}) and (\ref{3.88}), we get
\begin{equation}
a_{111}(a_{11}-a_{22})=a_{222}(a_{11}-a_{22})=0.
\end{equation}

If $a_{11}=a_{22}$, then by (\ref{3.86}), we have
\begin{equation}
|A|^2=\frac{H^2}{2}\leq \frac{2+\lambda^2+\lambda \sqrt{\lambda^2+4}}{4}.
\end{equation}
By Theorem \ref{3.45}, this implies $\Sigma$ is a round sphere.

If $a_{111}=a_{222}=0$, then $\abs{\nabla A}^2=0$.  Hence, 
\begin{equation}
\Big(\frac{1}{2}-\abs{A}^2\Big)\abs{A}^2=\lambda\langle A^2,A\rangle.
\end{equation}
Thus, we have 
\begin{equation}
\Big(|A|^2-\frac{1}{2}\Big)|A|^2=-\lambda\langle A^2,A\rangle \leq \lambda |A|^3.
\end{equation}
Therefore,
\begin{equation}
|A|^2 \leq \frac{1+\lambda^2+\lambda\sqrt{\lambda^2+2}}{2}.
\end{equation}
Combining this with (\ref{3.82}) and (\ref{3.85}) gives
\begin{equation}
(\lambda^2+1)\text{Area}(\Sigma) \leq \int_\Sigma H^2 \leq \frac{1+\lambda^2+\lambda\sqrt{\lambda^2+2}}{2} \text{Area}(\Sigma)+8\pi (1-g).
\end{equation} 
Observe that 
\begin{equation}
\lambda^2+1> \frac{1+\lambda^2+\lambda\sqrt{\lambda^2+2}}{2},
\end{equation}
then the genus $g=0$.
By Corollary \ref{3.81}, we conclude that $\Sigma$ is a round sphere. This completes the proof.

\end{proof}

\begin{remark}
Note that our method does not apply to higher dimensions. It is desirable that one may remove the conditions of closeness and $\lambda\geq 0$ to prove that any $\lambda$-hypersurface $\Sigma^2\subset \mathbb{R}^3$ with $|A|=constant$ is a generalized cylinder. 
\end{remark}

\section{Embedded $\lambda$-hypersurfaces in $\mathbb{R}^2$}
In this section, we will follow the argument in \cite{MC1} to show that any $\lambda$-hypersurface ($\lambda$-curve) in $\mathbb{R}^2$  with $\lambda \geq 0$ must either be a line or a round circle, i.e., Theorem \ref{1.7}.

\begin{proof}[Proof of Theorem \ref{1.7}]
Suppose $s$ is an arclength parameter of $\gamma$, then the curvature is $H=-\langle \nabla_{\gamma'}\gamma', \mathbf{n}\rangle$. Note that $\nabla_{\gamma'}\mathbf{n}=H\gamma'$, so we have
\begin{equation}\label{4.100}
2H'=\nabla_{\gamma'}\langle x,\mathbf{n} \rangle=H\langle x,\gamma' \rangle.
\end{equation}

If at some point $H=0$, then $H'=0$. By the uniqueness theorem of ODE, we conclude that $H\equiv0$, and, thus, $\gamma$ is just a line.
Therefore, we may assume that $H$ is always nonzero and possibly reversing the orientation of the curve to make $H>0$, i.e., $\gamma$ is strictly convex.

Differentiating $\abs{x}^2$ gives
\begin{equation}
(|x|^2)'=2\langle x,\gamma' \rangle=4\frac{H'}{H}.
\end{equation}
Thus $H=Ce^{\frac{|x|^2}{4}}$ for some constant $C>0$. 

Since the curve is strictly convex, we introduce a new variable $\theta$ by $\theta=\arccos\langle \mathbf{E}_1, n \rangle$. 

Differentiating with respect to the arclength parameter gives 
\begin{equation}
\partial_s \theta=-H,
\end{equation}

\begin{equation}
H_\theta=-\frac{H'}{H}=-\frac{\langle x,\gamma'\rangle}{2},
\end{equation}
and
\begin{equation}\label{4.105}
H_{\theta\theta}=\frac{\partial_s H_\theta}{-H}=\frac{1-2H(H-\lambda)}{2H}=\frac{1}{2H}-H+\lambda.
\end{equation}
Multiplying both sides of the above equation by $2H_\theta$, we get 
\begin{equation}
\partial_\theta(H_\theta^2+H^2-\log H-2\lambda H)=0.
\end{equation}
Therefore, the quantity 
\begin{equation}
E=H_\theta^2+H^2-\log H-2\lambda H
\end{equation}
is a constant.

Consider the function $f(t)=t^2-\log t-2\lambda t$, $t>0$. It is easy to verify that $f(t)\geq f(\frac{\lambda+\sqrt{\lambda^2+2}}{2})$. Hence, $E\geq f(\frac{\lambda+\sqrt{\lambda^2+2}}{2})$. 

If $E=f(\frac{\lambda+\sqrt{\lambda^2+2}}{2})$, then $H$ is constant and $\gamma$ must be a round circle.

Now we assume that $E> f(\frac{\lambda+\sqrt{\lambda^2+2}}{2})$. Note that $H=Ce^{\frac{\abs{x}^2}{4}}$ and $H\leq \abs{x/2}+\abs{\lambda}$. Then $H$ has an upper bound and $\abs{x}$ is bounded. By the embeddedness and completeness of $\gamma$, we conclude that $\gamma$ must be closed, simple and strictly convex.

If $\gamma$ is not a round circle, then we consider the critical points of the curvature $H$. By our assumption that $E> f(\frac{\lambda+\sqrt{\lambda^2+2}}{2})$, when $H_\theta=0$, we have $H_{\theta\theta}=\frac{1}{2H}-H+\lambda\neq 0$. So the critical points are not degenerate. By the compactness of the curve, they are finite and isolated.

Without loss of generality, we may assume $H(0)=H_{max}$ and $H(\bar{\theta})$ is the first subsequent critical point of $H$ for $\bar{\theta}>0$. Combining the fact that the curvature is strictly decreasing in the interval [0, $\bar{\theta}$] with  the second-order ODE of the function $H$ is symmetric with respect to $\theta=0$ and $\theta=\bar{\theta}$,  we conclude that $H(\bar{\theta})$  must be the minimum of the curvature. 

By the four-vertex theorem, we know that $\gamma$ has at least four pieces like the one described above. Since our curve is closed and embedded, the curvature $H$ is periodic with period $T<\pi$ and $\frac{T}{2}=\bar{\theta}$.

Next, we will evaluate an integral to produce a contradiction.

Since $H_{\theta\theta}=\frac{1}{2H}-H+\lambda$,  we have 
\begin{equation}
(H^2)_{\theta\theta\theta}+4(H^2)_\theta=\frac{2H_\theta}{H}+6\lambda H_\theta.
\end{equation}
Now we consider the following integral
\begin{equation}
2 \int_0^\frac{T}{2}   \sin2\theta \frac{H_\theta}{H} d\theta =\int_0^\frac{T}{2}    \sin2\theta \Big[(H^2)_{\theta\theta\theta}+4(H^2)_\theta-6\lambda H_\theta\Big] d\theta.
\end{equation}
Integration by parts gives 
\begin{equation}
\begin{split}
2 \int_0^\frac{T}{2}   \sin2\theta \frac{H_\theta}{H} d\theta 
  &=\sin 2\theta (H^2)_{\theta\theta}\big|_0^\frac{T}{2}-2\int_0^\frac{T}{2} \cos 2\theta  (H^2)_{\theta\theta} d\theta + 4\int_0^\frac{T}{2}   \sin2\theta (H^2)_\theta d\theta
\\ &-6 \lambda  \int_0^\frac{T}{2}   \sin2\theta H_\theta d\theta  \\&=2\sin T\Big[H^2_\theta(\frac{T}{2})+H(\frac{T}{2})H_{\theta\theta}(\frac{T}{2})\Big]-2\cos2\theta(H^2)_\theta \big |_0^\frac{T}{2} \\ &-6 \lambda  \int_0^\frac{T}{2}   \sin2\theta H_\theta d\theta \\ &=2\sin T H(\frac{T}{2})H_{\theta\theta}(\frac{T}{2})-6 \lambda  \int_0^\frac{T}{2}   \sin2\theta H_\theta d\theta.
\end{split} 
\end{equation}
By (\ref{4.105}) and $H_\theta(0)=H_\theta (\frac{T}{2})=0$, we get 
\begin{equation}\label{4.110}
2 \int_0^\frac{T}{2}   \sin2\theta \frac{H_\theta}{H} d\theta =2\sin T\Big[\frac{1}{2}-H^2(\frac{T}{2})+\lambda H(\frac{T}{2})\Big]-6 \lambda  \int_0^\frac{T}{2}   \sin2\theta H_\theta d\theta.
\end{equation}
Since $H$ is decreasing from 0 to $\frac{T}{2}$ and $\sin2\theta$ is nonnegative, the left-hand side of (\ref{4.110}) is nonpositive. For the right-hand side, the first term is nonnegative since $H(\frac{T}{2})$ is a minimum, and $\lambda\geq 0$ implies the second term is nonpositive. So the right-hand side of (\ref{4.110}) is nonnegative, and this gives a contradiction. Therefore, we conclude that $\gamma$ is a round circle.

\end{proof}

\begin{remark}
For the noncompact case, we do not need the condition $\lambda\geq 0$ to prove it is a line, and we do need $\lambda \geq 0$ for the closed case.  When $\lambda <0$, there exist some embedded $\lambda$-curves which are not round circles.
\end{remark}

\section{A Bernstein type theorem for $\lambda$-hypersurfaces}
The aim of this section is to prove Theorem \ref{1.5} which generalizes Ecker and Huisken's result \cite{EH1}. The key ingredient is that for a $\lambda$-hypersurface $\Sigma$, the function $\langle v, \mathbf{n}\rangle $ is an eigenfunction of the operator $L$ with eigenvalue 1/2, where $v \in \mathbb{R}^{n+1}$ is any constant vector. Note that the result is also true for  self-shrinkers. This eigenvalue result was also obtained by McGonagle and Ross \cite{MR1}. 

\begin{lemma}\label{5.180}
If $\Sigma\subset \mathbb{R}^{n+1}$ is a $\lambda$-hypersurface, then for any constant vector $v\in\mathbb{R}^{n+1}$, we have 
\begin{equation*}
L\langle v, \mathbf{n} \rangle =\frac{1}{2}\langle v, \mathbf{n} \rangle.
\end{equation*}

\end{lemma}
\begin{proof}
Set $f=\langle v, \mathbf{n} \rangle$. Working at a fixed point $p$ and choosing $e_i$ to be a local orthonormal frame, we have
\begin{equation*}
\nabla_{e_i}f=\langle v, \nabla_{e_i} \mathbf{n} \rangle =-a_{ij} \langle v, e_j \rangle.
\end{equation*}
Differentiating again and using Codazzi equation gives that
\begin{equation*}
\nabla_{e_k} \nabla_{e_i}=-a_{ijk}\langle v, e_j \rangle-a_{ij}a_{jk}\langle v, \mathbf{n} \rangle.
\end{equation*}
Therefore, 
\begin{equation}\label{5.182}
\Delta f=\langle v, \nabla H \rangle -\abs{A}^2 f.
\end{equation}
Using the equation of $\lambda$-hypersurfaces, we have
\begin{equation}\label{5.183}
\langle v, \nabla H \rangle=\langle v, -\frac{1}{2}a_{ij} \langle x, e_j \rangle e_i \rangle =\frac{1}{2}\langle x, \nabla f \rangle.
\end{equation}
Combining (\ref{5.182}) and (\ref{5.183}), we obtain that 
\begin{equation}
Lf=\Delta f-\frac{1}{2}\langle x, \nabla f \rangle+\Big(\frac{1}{2}+\abs{A}^2\Big)f=\frac{1}{2}f.
\end{equation}
\end{proof}

We are now in the position to prove Theorem \ref{1.5}.
\begin{proof}[Proof of Theorem \ref{1.5}]
Since $\Sigma$ is an entire graph, we can find a constant vector $v$ such that $f=\langle v,\mathbf{n} \rangle>0$. Let $u=1/f$. Then we have
\begin{equation}
\nabla u=-\frac{\nabla f}{f^2}\,\,\,\text{and}\,\,\, \Delta u=-\frac{\Delta f}{f^2}+\frac{2 |\nabla f|^2}{f^3}.
\end{equation}
By Lemma \ref{5.180}, we can easily get 
\begin{equation}
\mathcal{L}u=|A|^2 u+\frac{2|\nabla u|^2}{u}.
\end{equation}
Since $\Sigma$ has polynomial volume growth, we get
\begin{equation}
\int_\Sigma \Big(\abs{A}^2 u+\frac{2\abs{\nabla u}^2}{u}\Big)e^{-\frac{|x|^2}{4}}=0.
\end{equation}
Therefore, $|A|=0$ and $\Sigma$ is a hyperplane in $\mathbb{R}^{n+1}$.
\end{proof}

\section{Appendix}
Following the notation of \cite{CM3}, we call the quantity
\begin{equation}
H-\frac{\langle x,\mathbf{n}\rangle}{2}
\end{equation}  
rescaled mean curvature. Instead of considering $\lambda$-hypersurfaces, we consider more general hypersurfaces that have nonnegative rescaled mean curvature, i.e.,
\begin{equation}\label{300}
H-\frac{\langle x,\mathbf{n}\rangle}{2}\geq 0.
\end{equation}
Notice that $\lambda$-hypersurfaces ($\lambda\geq 0$) are just special cases of (\ref{300}). Closed hypersurfaces satisfying (\ref{300}) are closely related to a conjecture in \cite{CM3} and they behave nicely under the rescaled mean curvature flow.
Recall that a one-parameter family of hypersurfaces $M_t \subset \mathbb{R}^{n+1}$ flows by rescaled mean curvature if
\begin{equation}
\partial_{t}x=-\Big(H-\frac{\langle x,\mathbf{n}\rangle}{2}\Big)\mathbf{n}.
\end{equation}

Rescaled mean curvature flow is the negative gradient flow for the $F$-functional and self-shrinkers are the stationary points for this flow.

Under the rescaled MCF, closed hypersurfaces satisfying (\ref{300}) have several nice properties. If $M_0\subset \mathbb{R}^{n+1}$ is a smooth closed hypersurface satisfying (\ref{300}), then we have the following:
\begin{itemize}
\item Nonnegative rescaled mean curvature, i.e., (\ref{300}) is preserved under rescaled MCF, just like mean convexity is preserved under MCF.
\item As long as $(H-\frac{\langle x,\mathbf{n}\rangle}{2})>0$ holds at least at one point of $M_0$, the rescaled MCF $M_t$ will develop a singularity in finite time.
\item If $M_0$ satisfies an entropy condition, $\lambda(M_0)<3/2$, then at a singularity, there is a multiplicity one tangent flow of the form $\mathbb{S}^k \times \mathbb{R}^{n-k}$ for some $k>0$.
\end{itemize}

Basically, these three properties give a classification of singularities for rescaled MCF starting from a closed hypersurface satisfying (\ref{300}) and with low entropy.

Using corresponding Simons type identity and the parabolic maximum principle, the first two properties are not hard to prove. The last property is highly technical, and the proof involves analysis of the singular part of a weak solution of the self-shrinker equation (an integral rectifiable varifold) by using theories of stationary tangent cones.

Combining the above properties and a perturbation result, the following theorem was obtained in \cite{CM3}.
\begin{theorem}[{\cite{CM3}}]\label{302}
Given $n$, there exists $\epsilon=\epsilon(n)>0$ so that if ${\Sigma}^{n} \subset \mathbb{R}^{n+1}$ is a closed self-shrinker not equal to the round sphere, then $\lambda(\Sigma) \geq \lambda(\mathbb{S}^n)+\epsilon$. Moreover, if $\lambda(\Sigma)\leq$ min\{$\lambda(\mathbb{S}^n),\frac{3}{2}$\}, then $\Sigma$ is diffeomorphic to $\mathbb{S}^n$.
\end{theorem}
The proof of this theorem implies that the round sphere minimizes entropy among not only all closed self-shrinkers, but also all closed hypersurfaces satisfying (\ref{300}). 

By Huisken's monotonicity formula, the entropy is monotone non-increasing under MCF and therefore the entropy of the initial hypersurface gives a bound of the entropy at all future singularities.  So the study of entropy will help us better understand the singularities of MCF.  A related conjecture is the following:
\begin{conjecture}[{\cite{CM3}}]\label{303}
Theorem \ref{302} holds with $\epsilon=0$ for any closed hypersurface $M^n$ with $n\leq 6$.
\end{conjecture}
This conjecture was recently proved by Bernstein and Wang \cite{BW1}.

\bibliographystyle{alpha}
\bibliography{2015}

\begin{thebibliography}{CIMW13}

\bibitem[AL86]{AbL}
U.~Abresch and J.~Langer.
\newblock The normalized curve shortening flow and homothetic solutions.
\newblock {\em JDG}, 23(2):175--196, 1986.

\bibitem[Ang92]{A}
S.B. Angenent.
\newblock Shrinking doughnuts.
\newblock In {\em Nonlinear diffusion equations and their equilibrium states, 3
  ({G}regynog, 1989)}, volume~7 of {\em Progr. Nonlinear Differential Equations
  Appl.}, pages 21--38. Birkh\"auser Boston, Boston, MA, 1992.

\bibitem[BW14]{BW1}
J.~Bernstein and L.~Wang.
\newblock A sharp lower bound for the entropy of closed hypersurfaces up to
  dimension six.
\newblock {\em arXiv:1406.2966}, 2014.

\bibitem[CdCK70]{CDK}
S.S. Chern, M.~do~Carmo, and S.~Kobayashi.
\newblock Minimal submanifolds of a sphere with second fundamental form of
  constant length.
\newblock In {\em Functional {A}nalysis and {R}elated {F}ields ({P}roc. {C}onf.
  for {M}. {S}tone, {U}niv. {C}hicago, {C}hicago, {I}ll., 1968)}, pages 59--75.
  Springer, New York, 1970.

\bibitem[Cha14]{CH1}
J.E. Chang.
\newblock One dimensional solutions of the $\lambda $-self shrinkers.
\newblock {\em arXiv:1410.1782}, 2014.

\bibitem[CIMW13]{CM3}
T.H. Colding, T.~Ilmanen, W.P. Minicozz{i~II}, and B.~White.
\newblock The round sphere minimizes entropy among closed self-shrinkers.
\newblock {\em JDG}, 95(1):53--69, 2013.

\bibitem[CL13]{CL1}
H.~Cao and H.~Li.
\newblock A gap theorem for self-shrinkers of the mean curvature flow in
  arbitrary codimension.
\newblock {\em Calc. Var. Partial Differential Equations}, 46(3-4):879--889,
  2013.

\bibitem[CM12]{CM1}
T.H. Colding and W.P. Minicozz{i~II}.
\newblock Generic mean curvature flow {I}: generic singularities.
\newblock {\em Ann. of Math.}, 175(2):755--833, 2012.

\bibitem[CM15]{CM16}
T.H. Colding and W.P. Minicozz{i~II}.
\newblock Uniqueness of blowups and {L}ojasiewicz inequalities.
\newblock {\em Ann. of Math. (2)}, 182(1):221--285, 2015.

\bibitem[COW14]{COW1}
Q.M. Cheng, S.~Ogata, and G.~Wei.
\newblock Rigidity theorems of $\lambda$-hypersurfaces.
\newblock {\em Comm. Anal. Geom., to appear, arXiv:1403.4123}, 2014.

\bibitem[CW14a]{CW1}
Q.M. Cheng and G.~Wei.
\newblock Complete $\lambda$-hypersurfaces of weighted volume-preserving mean
  curvature flow.
\newblock {\em arXiv:1403.3177}, 2014.

\bibitem[CW14b]{CW2}
Q.M. Cheng and G.~Wei.
\newblock The gauss image of $\lambda$-hypersurfaces and a {B}ernstein type
  problem.
\newblock {\em arXiv:1410.5302}, 2014.

\bibitem[DX14]{DX1}
Q.~Ding and Y.L. Xin.
\newblock The rigidity theorems of self-shrinkers.
\newblock {\em Trans. Amer. Math. Soc.}, 366(10):5067--5085, 2014.

\bibitem[EH89]{EH1}
K.~Ecker and G.~Huisken.
\newblock Mean curvature evolution of entire graphs.
\newblock {\em Ann. of Math. (2)}, 130(3):453--471, 1989.

\bibitem[Gua14]{G1}
Q.~Guang.
\newblock Self-shrinkers with second fundamental form of constant length.
\newblock {\em arXiv:1405.4230}, 2014.

\bibitem[HW54]{HW1}
P.~Hartman and A.~Wintner.
\newblock Umbilical points and {$W$}-surfaces.
\newblock {\em Amer. J. Math.}, 76:502--508, 1954.

\bibitem[KM14]{KM}
S.~Kleene and N.M. M{\o}ller.
\newblock Self-shrinkers with a rotational symmetry.
\newblock {\em Trans. Amer. Math. Soc.}, 366(8):3943--3963, 2014.

\bibitem[Law69]{LS69}
H.B. Lawson.
\newblock Local rigidity theorems for minimal hypersurfaces.
\newblock {\em Ann. of Math. (2)}, 89:187--197, 1969.

\bibitem[Man11]{MC1}
C.~Mantegazza.
\newblock {\em Lecture notes on mean curvature flow}, volume 290 of {\em
  Progress in Mathematics}.
\newblock Birkh\"auser/Springer Basel AG, Basel, 2011.

\bibitem[M{\o}l11]{M1}
N.M. M{\o}ller.
\newblock Closed self-shrinking surfaces in $\mathbf{R}^3$ via the torus.
\newblock {\em arXiv:1111.7318}, 2011.

\bibitem[MR15]{MR1}
M.~McGonagle and J.~Ross.
\newblock The hyperplane is the only stable, smooth solution to the
  isoperimetric problem in gaussian space.
\newblock {\em Geom. Dedicata}, 178(1):277--296, 2015.

\bibitem[Ngu09]{N1}
X.H. Nguyen.
\newblock Construction of complete embedded self-similar surfaces under mean
  curvature flow. {I}.
\newblock {\em Trans. Amer. Math. Soc.}, 361(4):1683--1701, 2009.

\bibitem[Wan11]{WL}
L.~Wang.
\newblock A {B}ernstein type theorem for self-similar shrinkers.
\newblock {\em Geom. Dedicata}, 151:297--303, 2011.

\end{thebibliography}

\end{document}